\documentclass[11pt]{article}
\usepackage{amsmath, amsthm, amscd, amssymb,tikz,cite,graphicx}
\usepackage{enumitem}
\usepackage{caption,subcaption,xcolor,booktabs,multirow,setspace,dsfont,todonotes}

\newtheorem{thm}{Theorem}[section]
\newtheorem{lem}[thm]{Lemma}

\newtheorem{clm}[thm]{Claim}

\newtheorem{ques}[thm]{Question}

\numberwithin{equation}{section}

\newcommand{\G}{\mathcal{G}}
\newcommand{\I}{\mathcal{I}}
\newcommand{\F}{\mathcal{F}}

\newcommand{\PS}{{\mathcal{P}_{4,5}}}
\newcommand{\FF}[2]{$(\F_{#1},\F_{#2})$-partition}

\newcommand{\Ftf}{{\FF{3}{4}}}

\newcommand{\DeltaDelta}[2]{$(\Delta_{#1},\Delta_{#2})$-partition}

\begin{document}

\title{Partitioning planar graphs without $4$-cycles and $5$-cycles into bounded degree forests}

\author{Eun-Kyung Cho\thanks{
Department of Mathematics, Hankuk University of Foreign Studies, Yongin-si, Gyeonggi-do, Republic of Korea.
 \texttt{ekcho2020@gmail.com}
}
\and Ilkyoo Choi\thanks{Corresponding author. Department of Mathematics, Hankuk University of Foreign Studies, Yongin-si, Gyeonggi-do, Republic of Korea.
\texttt{ilkyoo@hufs.ac.kr}
}
\and Boram Park\thanks{
Department of Mathematics, Ajou University, Suwon-si, Gyeonggi-do, Republic of Korea.
\texttt{borampark@ajou.ac.kr}
}
} 

\date\today

\maketitle

\begin{abstract}
In 1976, Steinberg conjectured that planar graphs without $4$-cycles and $5$-cycles are $3$-colorable.
This conjecture attracted numerous researchers for about 40 years, until it was  recently disproved by Cohen-Addad et al. (2017). 
However, coloring planar graphs with restrictions on cycle lengths is still an active  area  of research, and the interest in this particular graph class remains. 
 
Let $G$ be a planar graph without $4$-cycles and $5$-cycles. 
For integers $d_1$ and $d_2$ satisfying $d_1+d_2\geq8$ and $d_2\geq d_1\geq 2$, it is known that $V(G)$ can be partitioned into two sets $V_1$ and $V_2$, where each $V_i$ induces a graph with maximum degree at most $d_i$.
Since Steinberg's Conjecture is false, a partition of $V(G)$ into two sets, where one induces an empty graph and the other induces a forest is not guaranteed.
Our main theorem is at the intersection of the two aforementioned research directions.
We prove that 
 $V(G)$ can be partitioned into two sets $V_1$ and $V_2$, where $V_1$ induces a forest with maximum degree at most $3$ and $V_2$ induces a forest with maximum degree at most $4$;
this is both a relaxation of Steinberg's conjecture and  a strengthening of results by Sittitrai and Nakprasit (2019) in a much stronger form.
\end{abstract}

\section{Introduction}\label{sec:intro}

We consider only finite simple graphs.
Given a graph $G$, let $V(G)$ and $E(G)$ denote the vertex set and edge set, respectively, of $G$.
A graph is {\it $k$-colorable} if its vertex set can be partitioned into $k$ color classes so that each color class induces an empty graph. 

The celebrated Four Color Theorem~\cite{AH1977, AHK1977} states that every planar graph is $4$-colorable.
Since there are planar graphs that are not $3$-colorable, finding sufficient conditions for planar graphs to be $3$-colorable is an active area of research. 
There is a vast literature in this direction, see an excellent survey by Borodin~\cite{Borodin2013}.
Before the Four Color Theorem was proved, Gr\"otzsch~\cite{Grotzsch1959} proved a result implying that a planar graph without $3$-cycles is $3$-colorable. 
In 1976, Steinberg (see~\cite{Steinberg1993}) conjectured that forbidding the next two cycle lengths should also be sufficient for a planar graph to be $3$-colorable.
Namely, Steinberg's Conjecture stated that every planar graph without $4$-cycles and $5$-cycles is $3$-colorable.

Steinberg's Conjecture attracted the interest of numerous researchers, but not much progress was made until Erd\H{o}s suggested the following relaxation: determine the minimum $k$ such that every planar graph without cycle lengths in $\{4, \ldots, k\}$ is $3$-colorable.
After progressions made by various researchers~\cite{AZ1991,Borodin1979,Borodin1996,SZ1995}, Borodin et al.~\cite{BGRS2005} showed $k=7$ is sufficient. 
Astonishingly, Steinberg's Conjecture was recently disproved~\cite{CHKLS2017},
yet, its legacy remains as studying the chromatic number of planar graphs with restrictions on cycle lengths is still an active area of research. 
In particular, the class of graphs considered in Steinberg's Conjecture, which are planar graphs without $4$-cycles and $5$-cycles, is still a popular domain of investigation.
Let $\PS$ denote the class of planar graphs without $4$-cycles and $5$-cycles.

For each $i\in\{1, \ldots, k\}$, let $\G_i$ be a class of graphs. 
Given a graph $G$, a {\it $(\G_1,\ldots,\G_k)$-partition} of $G$ is a  partition of its vertex set into $k$ sets $V_1, \ldots, V_k$ such that  $V_i$ induces a graph in $\G_i$ for each $i\in\{1, \ldots, k\}$.
Let $\F_d$ and $\Delta_d$ denote the class of forests and graphs, respectively, with maximum degree at most $d$.
Note that $\F_d  \subseteq \Delta_d$, and 
 equality holds only if and only if $d\in\{0, 1\}$.
 In particular, denote $\F_0$ by $\I$, which are classes of empty graphs and forests, respectively. 
Using this notation, the Four Color Theorem is equivalent to the statement that every planar graph has an $(\I,\I,\I,\I)$-partition.  
See Table~\ref{table:partition} for a summary of selected related results for such partitions of planar graphs with restrictions on girth. 

\begin{table}[h]
\begin{center}
\small{
\begin{tabular}{ | c | c | c | }
\hline
Classes & Partitions & References  \\
\hline\hline
 & $(\I, \I, \I, \I) $ & The Four Color Theorem \cite{AH1977, AHK1977}\\
 & $(\I, \F, \F)$ & Borodin \cite{Borodin1976}\\
Planar graphs & $(\F_{2}, \F_{2}, \F_{2})$ & Poh \cite{Poh1990} \\
 & $(\Delta_{2}, \Delta_{2}, \Delta_{2})$ &  Cowen, Cowen, Woodall~\cite{CCW1986} \\
 & no $(\Delta_{1}, \Delta_{d_1}, \Delta_{d_2})$ & Choi and Esperet \cite{CE2019} \\
\hline
 & $(\I, \I, \I)$ & Gr\"{o}tzsch \cite{Grotzsch1959} \\
Planar graphs with girth $4$ & $(\F_{5}, \F)$ & Dross, Montassier, Pinlou \cite{DMP2017}\\
 & no $(\Delta_{d_1}, \Delta_{d_2})$ & Montassier, Ochem \cite{MO2013} \\
\hline
 & $(\I, \F)$ & Borodin and Glebov \cite{BG2001} \\
Planar graph with girth $5$ & $(\Delta_1, \Delta_{10})$ & Choi et al.  \cite{CCJS2017} \\
 & $(\Delta_2, \Delta_{6})$ & Borodin, Kostochka \cite{BK2014} \\
 & $(\Delta_3, \Delta_{4})$ & Choi, Yu, Zhang  \cite{CYZ2019} \\
\hline
Planar graph with girth $6$ & no $(\I, \Delta_{d})$ & Borodin et al. \cite{BIMOR2010}\\
\hline
\end{tabular}}
\caption{Related results on partitions of planar graphs with girth restrictions}\label{table:partition}
\end{center}
\end{table}

The situation where each color class induces a graph with bounded degree is also known as {\it defective coloring} in the literature. 
Sittitrai and Nakprasit~\cite{SN2018} proved that there does not exist an integer $k$, where graphs in $\PS$ have a \DeltaDelta{1}{k}.
They also proved that graphs in $\PS$ have a \DeltaDelta{4}{4}, a \DeltaDelta{3}{5}, and a \DeltaDelta{2}{9}.
Liu and Lv~\cite{LL2019} improved the last result by showing that every graph in $\PS$ has a \DeltaDelta{2}{6}.

\begin{thm}[\cite{SN2018,LL2019}]\label{SN-44-35}
Every planar graph without $4$-cycles and $5$-cycles has a \DeltaDelta{4}{4}, a \DeltaDelta{3}{5}, and a \DeltaDelta{2}{6}.
\end{thm}

We focus on the case when one of the color classes is more restrictive, namely, forests with bounded degree. 
Since Steinberg's Conjecture is false, this implies that there is a graph in $\PS$ that does not have an $(\I, \F)$-partition.
By relaxing the first color class from $\I$ to $\F_d$ for some $d$, it is natural to ask if each graph in $\PS$ has an $(\F_d, \F)$-partition or not.
Surprisingly, our main result states that it is even possible to bound the maximum degree of the second color class as soon as $d=3$. 

We remark that our main result is at the intersection of the two aforementioned research directions; 
it is both a relaxation of Steinberg's Conjecture and also a strengthening of results by Sittitrai and Nakprasit in a much stronger form.
Our main result is the following: 

\begin{thm}\label{thm:main}
Every planar graph without $4$-cycles and $5$-cycles has an \Ftf.
\end{thm}

Recall that in all previously known results proving that every graph in $\PS$ has a \DeltaDelta{d_1}{d_2}, the sum satisfies $d_1+d_2\geq 8$.
Theorem~\ref{thm:main} provides the first result showing that such a partition is possible even when the sum is less than 8, even in the stronger form since each color class has no cycles.

We use the discharging method to prove Theorem~\ref{thm:main}. 
The proof of Theorem~\ref{thm:main} is split into Sections~\ref{sec:f34} and~\ref{sec:rc-proofs}. 
Section~\ref{sec:f34} lays out the discharging rules and reducible configurations. 
The proofs of the reducible configurations are presented in Section~\ref{sec:rc-proofs}.
In Section~\ref{sec:future}, some future research directions are suggested.

We end this section with some definitions  used throughout the paper. 
A {\it $k$-vertex}, {\it $k^{+}$-vertex}, and {\it $k^{-}$-vertex} are a vertex of degree $k$, at least $k$, and at most $k$, respectively. 
A {\it $k$-neighbor} of a vertex is a neighbor that is a $k$-vertex; 
{\it $k^+$-neighbor} and {\it $k^-$-neighbor} are defined analogously.  
Similarly, a {\it $k$-face}, {\it $k^+$-face}, and {\it $k^-$-face} are also defined. 
A {\it $(d_1, d_2, d_3)$-face} is a $3$-face, where the vertices on the face have degrees $d_1, d_2, d_3$. 
If a vertex $v$ is adjacent to a $3$-vertex $u$ on a $3$-face $f$, but $v$ is not on $f$, then $f$ is a {\it pendent face} of $v$ and $v$ is the {\it pendent neighbor} of $u$.
Note that throughout the figures in the paper, the degree of a solid (black) vertex is the number of incident edges drawn in the figure, whereas a hollow (white) vertex indicates a $2^+$-vertex.


\section{Proof of Theorem~\ref{thm:main}: \Ftf}\label{sec:f34}

Let $G$ be a counterexample to Theorem~\ref{thm:main} with the minimum number of vertices, and fix a plane embedding of $G$. 
Let $F(G)$ be the set of faces of $G$.
In this section, we provide a list of subgraphs where each subgraph does not appear in $G$; each such subgraph is referred to as a {\it reducible configuration}. 
We then lay out the discharging rules to reach a contradiction.   
The proofs of the reducible configurations are in Section~\ref{sec:rc-proofs}.

\subsection{Reducible Configurations}

We first define the following sets. 
\begin{eqnarray*}
W_2 &=&\{v\in V(G)\mid \text{$v$ is a $2$-vertex not on a $3$-face}  \}\\
F_2 &=&\{f\in F(G)\mid \text{$f$ is a $3$-face incident with a $2$-vertex}\}\\
F_3&=& \{f\in F(G)\mid \text{$f$ is a $3$-face incident with a $3$-vertex}\}
\end{eqnarray*}
A $3$-face in $F_3$ is {\it terrible} if it is a $(3,3,d)$-face with a pendent $4^-$-neighbor. \begin{figure}[h!]
  \centering
  \includegraphics[width=2.3cm,page=3]{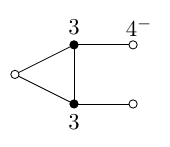}
  \caption{A terrible face}
  \label{fig:ter_face}
\end{figure}
See Figure~\ref{fig:ter_face}. For $d \in \{6,7,8\}$, a $d$-vertex is \textit{bad} if it is incident with exactly $(d-5)$ terrible faces, one non-terrible face in $F_2 \cup F_3$, and all neighbors not on those $(d-4)$ $3$-faces are $3^-$-vertices. 
Note that a bad vertex has at most one $4^+$-neighbor.
See Figure~\ref{fig:bad6,7,8}. Let 
$$ F^*_2 = \{ f \in F_2 \mid f \text{ is incident with either a $5$-vertex or a bad vertex} \}.$$

\begin{figure}[h!]
  \centering \includegraphics[width=0.81\textwidth]{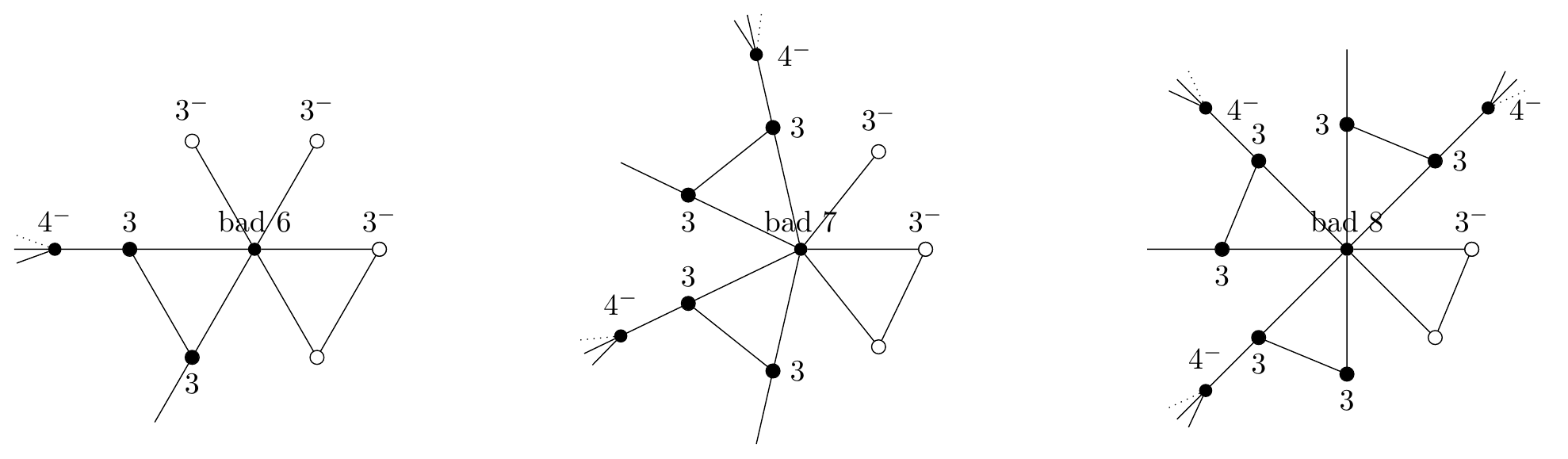}
  \caption{A bad $6$-, $7$-, $8$-vertex}  
  \label{fig:bad6,7,8}
\end{figure}
 
The following is a list of reducible configurations. 
We postpone the proofs to Section~\ref{sec:rc-proofs}. 

\begin{enumerate}[label=\textbf{[C\arabic*]}]
\item\label{rc-7-face}
(Lemma~\ref{lem:rc-1vx} (i))
A $7^+$-face $f$ incident with $(\deg(f)-5)$ $2$-vertices on $F_2$.
\item\label{rc-1vx} (Lemma~\ref{lem:rc-1vx} (ii)) A $1^-$-vertex. 
\item\label{rc-2vx4} (Lemma~\ref{lem:rc-1vx} (iii)) A $2$-vertex with a $4^-$-neighbor. 
\item\label{rc-bad2face256} (Lemma~\ref{lem:rc-1vx} (iv)) A $(2, 5, 6^-)$-face.
\item\label{rc-W3pen4} (Lemma~\ref{lem:rc-1vx} (v)) 
A $(3,5^-,5^-)$-face with a pendent $4^-$-neighbor.
\item\label{rc-5vxa} (Lemma~\ref{lem:5-vertex} (i)) A $5$-vertex with only $3^-$-neighbors, where one of them is in $W_2$.
\item\label{rc-5vxb} (Lemma~\ref{lem:5-vertex} (ii))
A $5$-vertex with five pendent $3$-faces, where four of them are $(3,5^-,5^-)$-faces.

\item\label{rc-bad2face257a} (Lemma~\ref{lem:rc-bad2face257} (i)) 
A bad vertex with only $6^-$-neighbors.
\item\label{rc-bad2face257b} (Lemma~\ref{lem:rc-bad2face257} (ii)) 
An $(F_2 \cup F_3)$-face with two bad vertices.
\item\label{rc-mediumd-3T3} (Lemma~\ref{lem:rc-bad2face257} (iii)) A $6$-vertex on three $3$-faces, one of which is a terrible face.
\item\label{rc-mediumd-4T3} (Lemma~\ref{lem:rc-bad2face257} (iv)) For $d\in\{6, \ldots, 10\}$, a $d$-vertex on $(d-4)$ terrible faces.
\item\label{rc-mediumd-5T3-6} (Lemma~\ref{lem:rc-bad2face257} (v)) 
For $d\in\{6,\ldots,10\}$, a $d$-vertex $v$ on $(d-5)$ terrible faces, where $v$ has only $3^-$-neighbors.
\item\label{rc-mediumT3H3} (Lemma~\ref{lem:rc-mediumT3H3}) For $d\in\{7, \ldots, 10\}$, a non-bad $d$-vertex $v$ on an $F^*_2$-face, where $v$ is on $(d-6)$ other $3$-faces, each of which is either terrible or in $F^*_2$.
\end{enumerate}


\subsection{Discharging}

In order to reach the final contradiction, we use the discharging technique.
To each vertex $v$ and each face $f$, let $2\deg(v)-6$ and $\deg(f)-6$ be its {\it initial charge} $\mu(v)$ and $\mu(f)$, respectively.
The total initial charge is negative, since Euler's formula implies
\[\sum_{v \in V(G)} \mu(v) + \sum_{f \in F(G)} \mu(f)=
\sum_{v \in V(G)} (2\deg(v)-6) + \sum_{f \in F(G)} (\deg(f)-6)
= -12<0.\]
We then redistribute the charge at the vertices and faces according to carefully designed {\it discharging rules}, which preserve the total charge sum.
After the discharging procedure, we end up with {\it final charge} $\mu^*(z)$ at each $z\in V(G)\cup F(G)$. 
We will prove that the final charge at each $z$ is non-negative, to conclude that the total final charge sum is non-negative. 
This is the final contradiction since the total charge sum was preserved. 
The discharging rules are as follows.
See Figure~\ref{fig:rules}.

\begin{figure}[h!]
  \centering
\includegraphics[width=\textwidth]{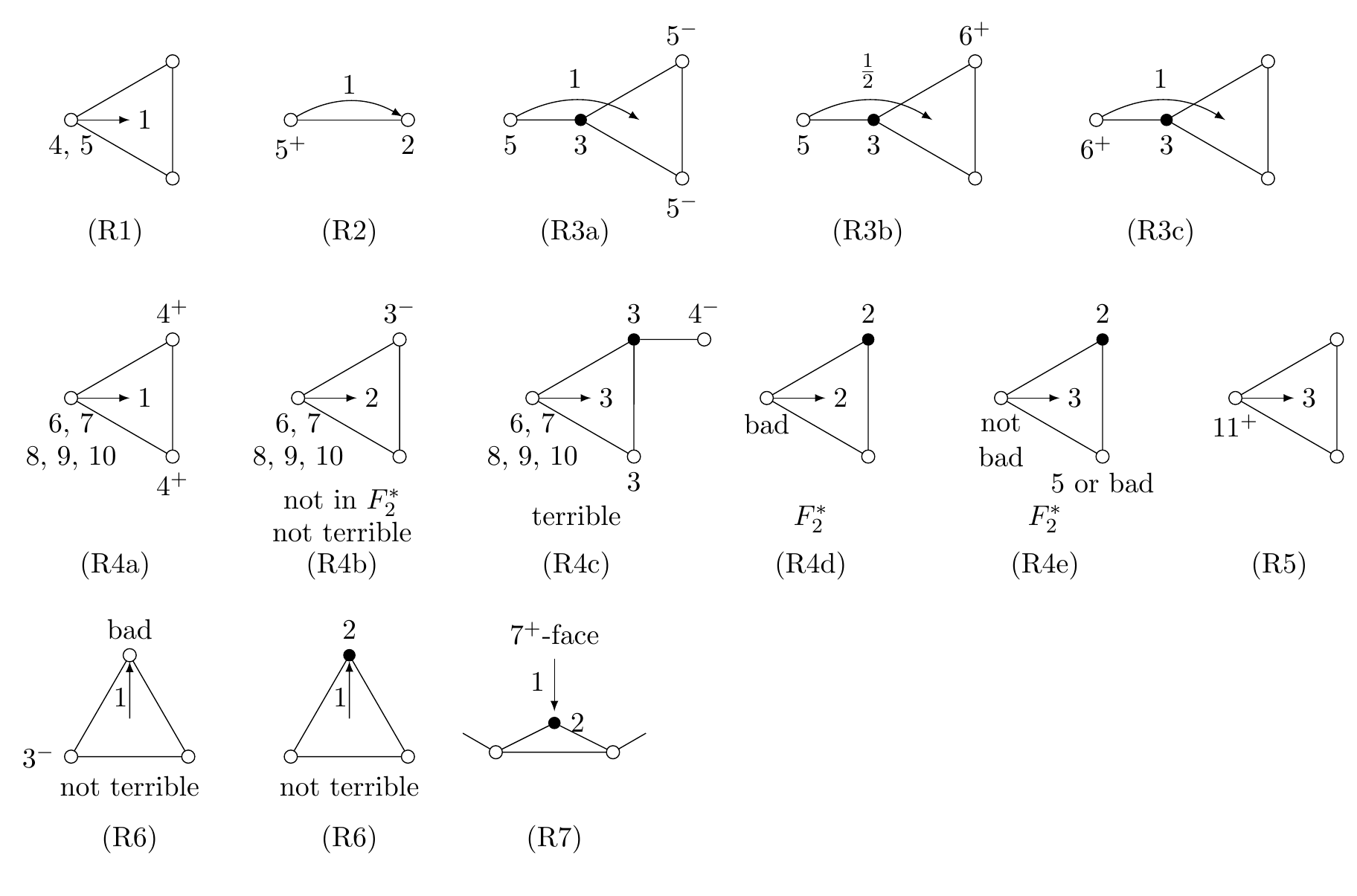}  
  \caption{An illustration of the discharging rules}
  \label{fig:rules}
\end{figure}

\begin{enumerate}[label=\textbf{R\arabic*}]
\item\label{rule-45} For $d\in\{4, 5\}$, each $d$-vertex sends charge 1 to each incident $3$-face. 
\item\label{rule-good2} Each $5^+$-vertex sends charge $1$ to each neighbor in $W_2$.
\item\label{rule-pendent} Let $v$ be a $5^+$-vertex.
\begin{enumerate}
    \item\label{rule-5_ch1} If $v$ is a $5$-vertex, then it sends charge $1$ to each pendent $(3,5^-,5^-)$-face.
    \item\label{rule-5_ch1/2} If $v$ is a $5$-vertex, then it sends charge $\frac{1}{2}$ to each pendent face that is incident with a $6^+$-vertex.
    \item\label{rule-6_pendent} If $v$ is a $6^+$-vertex, then it sends charge $1$ to each pendent face.
\end{enumerate}

\item Let $v$ be a $d$-vertex, where $d\in\{6, \ldots, 10\}$. 
\begin{enumerate}
    \item\label{rule-6_ch1} $v$ sends charge $1$ to each incident $3$-face not in $F_2 \cup F_3$.

    \item\label{rule-6_ch2} $v$ sends charge $2$ to each incident $(F_2 \cup F_3)$-face that is neither terrible nor in $F^*_2$.
    
    \item\label{rule-terrible} $v$ sends charge $3$ to each incident terrible face. 
     
    \item\label{rule-f2_bad} If $v$ is bad, then $v$ sends charge $2$ to each incident $F^*_2$-face.
    
    \item\label{rule-f2_nonbad} If $v$ is not bad, then $v$ sends charge $3$ to each incident $F^*_2$-face.
\end{enumerate}

\item\label{rule-11} Each $11^{+}$-vertex sends charge $3$ to each incident $3$-face.


\item\label{rule-bad3face} Each non-terrible $(F_2\cup F_3)$-face sends charge $1$ to each incident $2$-vertex and to each incident bad vertex.

\item\label{rule-7face} Each $7^{+}$-face sends charge $1$ to each incident $2$-vertex on $F_2$.
\end{enumerate}

It remains to check the final charge of each vertex and each face.

\begin{clm}\label{clm:vx-f}
Each vertex $v$ has non-negative final charge. 
\end{clm}
\begin{proof}
Note that a vertex $v$ is on at most $\lfloor \frac{\deg(v)}{2} \rfloor$ $3$-faces since $G$ has no $4$-cycles.
By \ref{rc-1vx}, each vertex of $G$ is a $2^+$-vertex.

Suppose that  $v$ is a $2$-vertex.
By \ref{rc-2vx4}, $v$ has two $5^+$-neighbors.
If $v \in W_2$, then $\mu^*(v)=(-2)+1\cdot 2 =0$ by \ref{rule-good2}. 
If $v \notin W_2$, then since $G$ has no $5$-cycles, $v$ is on both a non-terrible $3$-face and a $7^+$-face by~\ref{rc-2vx4} and~\ref{rc-bad2face256}, so $\mu^*(v)=(-2)+1\cdot2 = 0$ by \ref{rule-bad3face} and \ref{rule-7face}.
If $v$ is a $3$-vertex, then $v$ is not involved in the discharging rules,  so $\mu^*(v)=\mu(v)=0$.
If $v$ is a $4$-vertex, then since $v$ is on at most two $3$-faces, $\mu^*(v)\ge 2-1\cdot 2 = 0$ by \ref{rule-45}.

Suppose that  $v$ is a $5$-vertex, so $v$ is on at most two $3$-faces.
By \ref{rc-5vxa} and \ref{rc-5vxb}, $v$ does not send charge $1$ five times, so either $\mu^*(v) \ge 4-1\cdot 4=0$ or $\mu^*(v) \ge 4- 1\cdot 3 - \frac{1}{2} \cdot 2 = 0$ by~\ref{rule-45}, \ref{rule-good2}, \ref{rule-5_ch1}, and \ref{rule-5_ch1/2}.

Suppose that  $v$ is a $d$-vertex for $d\in\{6,\ldots,10\}$.
If $v$ is a bad vertex, then
$\mu^*(v)\ge (2d-6)- 3 \cdot (d-5) -2 \cdot 1-1 \cdot (d-2(d-4))+1 \cdot 1 = 0$
by \ref{rule-good2}, \ref{rule-6_pendent}, \ref{rule-6_ch2}, \ref{rule-terrible}, \ref{rule-f2_bad}, and \ref{rule-bad3face}.
Note that the non-terrible face in $F_2\cup F_3$ that is incident with $v$ sends charge $1$ to $v$ by~\ref{rule-bad3face}. 

Suppose that  $v$ is not a bad vertex.
If $v$ is on at most $(d-6)$ $3$-faces that are either terrible or $F^*_2$-faces, then
$\mu^*(v) \ge (2d-6) - 3\cdot (d-6) - 1 \cdot (d-2(d-6))=0$ by \ref{rule-good2}, \ref{rule-6_pendent}, \ref{rule-6_ch1}, \ref{rule-6_ch2}, \ref{rule-terrible}, and \ref{rule-f2_nonbad}.
Suppose that  $v$ is on at least $(d-5)$ $3$-faces that are either terrible or $F^*_2$-faces.
By~\ref{rc-bad2face256}, \ref{rc-bad2face257a} when $d=6$, and by~\ref{rc-mediumT3H3} when $d\in\{7, \ldots, 10\}$, $v$ is not on an $F^*_2$-face. 
By \ref{rc-mediumd-4T3}, $v$ is on exactly $(d-5)$ terrible faces.
By \ref{rc-mediumd-5T3-6}, we conclude that $v$ has a $4^+$-neighbor $u$.
Note that $v$ is on at most one non-terrible $3$-face since $d-2(d-5)\leq 3$ when $d\in\{7, \ldots, 10\}$ and by~\ref{rc-mediumd-3T3} when $d=6$.

If $uv$ is not on a $3$-face, then $\mu^*(v) \ge (2d-6) - 3\cdot (d-5) -1\cdot(d-1-2(d-5)) = 0$
by \ref{rule-good2}, \ref{rule-6_pendent}, \ref{rule-6_ch1}, \ref{rule-6_ch2}, and \ref{rule-terrible}.
If $uv$ is on a $3$-face $f$, then 
since $v$ is not a bad vertex, $f$ is a $(4^+,4^+,4^+)$-face. 
Thus $\mu^*(v) \ge (2d-6) - 3\cdot (d-5) -1\cdot 1-1\cdot(d-2(d-4)) = 0$ by \ref{rule-good2}, \ref{rule-6_pendent}, \ref{rule-6_ch1}, and \ref{rule-terrible}.

If $v$ is a $d$-vertex for  $d \geq 11$, then $\mu^*(v) \ge (2d-6)- 3 \cdot \lfloor \frac{d}{2} \rfloor -1\cdot (d-2 \lfloor \frac{d}{2} \rfloor)=  \left\lceil \frac{d}{2}\right\rceil-6\ge 0$ by~\ref{rule-good2}, \ref{rule-6_pendent}, and \ref{rule-11}.
\end{proof}

\begin{clm}\label{clm:face-f}
Each face $f$ has non-negative final charge. 
\end{clm}

\begin{proof}
If $f$ is a $7^+$-face, then by \ref{rc-7-face}, $f$ is incident with at most $(\deg(f)-6)$ $2$-vertices that are on $F_2$, so $\mu^{*}(f) \geq (\deg(f)-6) - 1\cdot(\deg(f)-6) = 0$ by \ref{rule-7face}.
A $6$-face $f$ is not involved in the discharging rules, so $\mu^*(f) = \mu(f) = 0$.
Note that $G$ has neither $4$-faces nor $5$-faces.

Suppose $f$ is a $3$-face $u_1u_2u_3$, where $\deg(u_1)\leq \deg(u_2)\leq \deg(u_3)$.
If $f$ is terrible, then by \ref{rc-W3pen4}, $f$ is incident with a $6^+$-vertex, which sends charge $3$ to $f$ by~\ref{rule-terrible} and~\ref{rule-11}, so $\mu^*(f)\geq (-3)+3=0$. 
Suppose $f$ is an $F^*_2$-face.  
By \ref{rc-bad2face257a} and \ref{rc-bad2face257b}, there is exactly one vertex $w$ on $f$ that is either a $5$-vertex or a bad vertex.
If $w$ is a $5$-vertex, then $\mu^{*}(f) \geq (-3) + 1 + 3 - 1 \cdot 1= 0$
by~\ref{rule-45},~\ref{rule-f2_nonbad}, and~\ref{rule-bad3face}.
If $w$ is bad, then $\mu^{*}(f) \geq (-3) + 2 + 3 -1 \cdot 2 = 0$ by~\ref{rule-f2_bad},~\ref{rule-f2_nonbad}, and~\ref{rule-bad3face}.
Now assume $f$ is neither terrible nor in $F^*_2$. 

If $u_1$ is a $4^+$-vertex, then $f$ is a $(4^+,4^+,4^+)$-face.
Each vertex on $f$ sends charge at least $1$ to $f$ by \ref{rule-45}, \ref{rule-6_ch1}, and \ref{rule-11}, so  $\mu^*(f) \ge (-3)+1 \cdot 3 = 0$. 

If $u_1$ is a $2$-vertex, then $f$ is not incident with a $5$-vertex or a bad vertex since $f\not\in F^*_2$. 
Also, by \ref{rc-2vx4}, $f$ is a $(2,6^+,6^+)$-face.
Each $6^+$-vertex on $f$ sends charge at least $2$ to $f$ by \ref{rule-6_ch2} and~\ref{rule-11}, and $f$ sends charge $1$ to $u_1$ by \ref{rule-bad3face}, so $\mu^{*}(f) \geq (-3) + 2 \cdot 2 -1\cdot 1= 0$.
 
Suppose that  $u_1$ is a $3$-vertex.
If $u_3$ is a $5^-$-vertex, then by \ref{rc-W3pen4}, each $3$-vertex on $f$ has a pendent $5^+$-neighbor, so $\mu^*(f)\ge (-3) + 1 \cdot 3=0$ by~\ref{rule-45},~\ref{rule-5_ch1}, and~\ref{rule-6_pendent}. 
So now assume $u_3$ is a $6^+$-vertex.

Suppose that  $u_2$ is a $3$-vertex, so $f$ is a $(3,3,6^+)$-face.
Since $f$ is not terrible, each of $u_1$ and $u_2$ has a pendent $5^+$-neighbor.
Moreover, by~\ref{rc-bad2face257a}, $u_3$ cannot be a bad vertex. 
Each pendent neighbor of $f$ sends charge at least $\frac{1}{2}$ to $f$ by~\ref{rule-pendent}, so $\mu^*(f)\geq (-3)+2 \cdot 1 +\frac{1}{2}\cdot 2=0$ by~\ref{rule-pendent},~\ref{rule-6_ch2},~and~\ref{rule-11}. 

Now suppose that  $u_2$ is a $4^+$-vertex, so $f$ is a $(3, 4^+, 6^+)$-face.
Then by \ref{rc-bad2face257a} and \ref{rc-bad2face257b}, there is at most one vertex $w$ on $f$ that is either a $5$-vertex or a bad vertex.

If $f$ is not incident with a bad vertex,
then $\mu^*(f) \ge (-3)+1\cdot1 +2\cdot1=0$ by~\ref{rule-45}, \ref{rule-6_ch2}, and \ref{rule-11}.
If $f$ is incident with a bad vertex, then by \ref{rc-bad2face257a}, $u_2$ is a $7^+$-vertex, so $\mu^*(f) \ge (-3)+2 \cdot 2-1\cdot1=0$ by \ref{rule-6_ch2}, \ref{rule-11}
and \ref{rule-bad3face}.
\end{proof}

By Claims~\ref{clm:vx-f} and~\ref{clm:face-f}, each vertex and each face has non-negative final charge.
Hence, the final charge sum is non-negative.


\section{Proofs for Reducible Configurations}\label{sec:rc-proofs}

In this section, we present the proofs of the reducible configurations. 
Recall that $G$ is a planar graph without $4$-cycles and $5$-cycles that does not have an \Ftf, but all its proper subgraphs have an \Ftf.

If $(A_3, A_4)$ is an $(\F_{3},\F_{4})$-partition of a proper induced  subgraph of a graph $G$, then $(A_3, A_4)$ is a {\it partial $(\F_{3},\F_{4})$-partition} of $G$.
Given a (partial) $(\F_{3},\F_{4})$-partition $(A_3, A_4)$ of a graph $G$ and a vertex $v$ of $G$,
a neighbor of $v$ in $A_i\cap N_G(v)$ is an {\it $A_i$-neighbor} of $v$, 
and we say $v$ is {\it $A_i$-saturated} if $v$ has $i$ $A_i$-neighbors.

For  $S\subset V(G)$, let $G-S$ denote the graph obtained from $G$ by deleting the vertices in $S$.
If $S=\{x\}$, then denote $G-S$ by $G-x$.
Likewise, in order to improve readability, we often drop the braces and commas to denote a set and use `$+$' for the set operation `$\cup$'.
For instance, given $A,B \subset V(G)$ and $x,y,z\in V(G)$, we use $A+x-y$, $A-z+xy$, and $A+B$ to denote $(A\cup\{x\})\setminus\{y\}$, $(A\setminus\{z\})\cup\{x,y\}$, and $A \cup B$, respectively.

\begin{lem}\label{lem:coloring}
If $(A_3,A_4)$ is an {\Ftf} of $G-v$, then the following holds:
\begin{itemize}
    \item[\rm(i)] The neighbors of $v$ cannot all be in the same part.
    \item[\rm(ii)] For $i \in \{3,4\}$, if $u$ is the only $A_i$-neighbor of $v$, then $u$ is $A_i$-saturated and has an $A_{7-i}$-neighbor. 
    That is, $u$ is an $(i + 2)^+$-vertex.
    \item[\rm(iii)]  For  $i \in \{3,4\}$, if an $(i + 2)$-vertex $u$ is the only $A_i$-neighbor of $v$, then $u$ has an $A_{7-i}$-saturated neighbor in $A_{7-i}$.
\end{itemize}
\end{lem}
\begin{proof}
Let $v$ be a vertex of $G$. 
By the minimality of $G$, there exists an {\Ftf} $(A_3, A_4)$ of $G-v$.
If every neighbor of $v$ is an $A_i$-neighbor 
for some $i\in\{3,4\}$, then putting $v$ in $A_{7-i}$ gives an {\Ftf} of $G$, which is a contradiction, so (i) holds.

Let $u$ be the only $A_i$-neighbor of $v$ for some $i\in\{3, 4\}$. 
Since adding $v$ to $A_i$ is not an {\Ftf} of $G$, this implies that $u$ is $A_i$-saturated.
Since moving $u$ from $A_i$ to $A_{7-i}$ and adding $v$ to $A_i$ is not an {\Ftf} of $G$, $u$ also has an $A_{7-i}$-neighbor $w$. 
Thus (ii) holds.
Moreover, if $u$ is an $(i+2)$-vertex, then $w$ is the only $A_{7-i}$-neighbor of $u$, so $w$ must be $A_{7-i}$-saturated, hence (iii) holds.
\end{proof}

\begin{lem}\label{lem:rc-1vx}
In the graph $G$, the following holds:
\begin{itemize}
\item[\rm(i)] There is no $7^+$-face $f$ incident with $(\deg(f)-5)$ $2$-vertices on $F_2$.
{\rm\ref{rc-7-face}}
\item[\rm(ii)] There is no $1^-$-vertex. {\rm{\ref{rc-1vx}}}
\item[\rm(iii)] There is no $2$-vertex with a $4^-$-neighbor. {\rm{\ref{rc-2vx4}}}
\item[\rm(iv)] There is no $(2,5,6^-)$-face.
{\rm\ref{rc-bad2face256}}
\item[\rm(v)] There is no  $(3,5^-,5^-)$-face with a pendent $4^-$-neighbor. {\rm\ref{rc-W3pen4}}
\end{itemize}
\end{lem}

\begin{proof}
If $G$ has a $7^+$-face $f$ incident with $(\deg(f)-5)$ $2$-vertices on $F_2$, then this implies $G$ has a $5$-face, which is a contradiction. 
Thus (i) holds. 
(ii) immediately follows from Lemma~\ref{lem:coloring} (i), and
(iii) immediately follows from Lemma~\ref{lem:coloring} (i) and (ii).

To show (iv), let $xyz$ be a $3$-face, where $x$ is a $2$-vertex.
By Lemma~\ref{lem:coloring} (i) and (ii), we may assume $y$ and $z$ are a $5^+$-vertex and a $6^+$-vertex, respectively. 
Suppose to the contrary that $xyz$ is a $(2, 5, 6)$-face. 
By the minimality of $G$, there exists an {\Ftf} $(A_3,A_4)$ of $G-x$.  
By Lemma~\ref{lem:coloring}, $y\in A_3$, $z\in A_4$, $y$ is $A_3$-saturated, and $z$ is $A_4$-saturated.
Now, $(A_3+xz-y,A_4+y-z)$ is an {\Ftf} of $G$, which is a contradiction. 

To show (v), let $xyz$ be a $(3,5^-,5^-)$-face, where $x$ is a $3$-vertex. Suppose to the contrary that the pendent neighbor $x'$ of $x$ is a $4^-$-vertex. 
By the minimality of $G$, there exists an {\Ftf} $(A_3,A_4)$ of $G-x$.
Since no neighbor of $x$ is a $6^+$-vertex, by Lemma~\ref{lem:coloring}, there is exactly one $A_3$-neighbor $w$ of $x$ that is $A_3$-saturated and has a $A_4$-saturated neighbor; this is impossible since $x'$ is a $4^-$-vertex, and $y$ and $z$ are $5^-$-vertices.
\end{proof}

\begin{lem}\label{lem:5-vertex}
In the graph $G$, the following holds:
\begin{itemize}
\item[\rm(i)] There is no $5$-vertex with only $3^-$-neighbors, where one of them is in $W_2$. {\rm\ref{rc-5vxa}}
\item[\rm(ii)] There is no $5$-vertex with five pendent $3$-faces, where 
 four of them are $(3,5^-,5^-)$-faces.
    {\rm\ref{rc-5vxb}}
\end{itemize}    
\end{lem}

\begin{proof} Let $v$ be a $5$-vertex with neighbors $x_1, \ldots, x_5$. 

To show (i), suppose to the contrary that $x_1\in W_2\cap N_G(v)$, and all other neighbors of $v$ are $3^-$-neighbors.
Let $y_1$ be the neighbor of $x_1$ that is not $v$.
See the left figure in Figure~\ref{fig:(f3,f4)_good2}.
By the minimality of $G$, there exists an {\Ftf} $(A_3,A_4)$ of $G-x_1$.
By Lemma~\ref{lem:coloring}, we know $v \in A_3$, $v$ is $A_3$-saturated, and $v$ has an $A_4$-saturated neighbor.
This is a contradiction since all neighbors of $v$ are $3^-$-vertices.

\begin{figure}[h!]
  \centering
  \includegraphics[height=2.8cm,page=8]{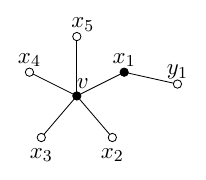} \qquad    \includegraphics[height=4.2cm,page=9]{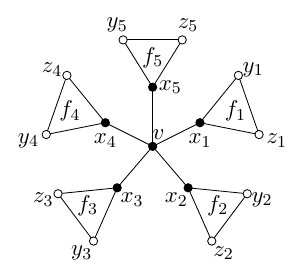}  
  \caption{A $5$-vertex $v$ with either a neighbor $x_1\in W_2$ or five pendent faces}
  \label{fig:(f3,f4)_good2}
\end{figure}

To show  (ii), assume that $v$ has five pendent $3$-faces $x_1y_1z_1, \ldots, x_5y_5z_5$. See the right figure in Figure~\ref{fig:(f3,f4)_good2}. Suppose to the contrary that four of those $3$-faces are $(3,5^-,5^-)$-faces.
By the minimality of $G$, there exists an {\Ftf} $(A_3, A_4)$ of $G-v$. 
Note that since every neighbor of $v$ is a $3$-vertex, no neighbor of $v$ is either $A_3$-saturated or $A_4$-saturated. 
Therefore, if there is a part $P$ with at most one neighbor of $v$, then putting $v$ in $P$ gives an {\Ftf} $(A_3, A_4)$ of $G$, which is a contradiction. 
Hence, each of $A_3$ and $A_4$ has at least two neighbors of $v$, and since $v$ is a $5$-vertex, there is a part with exactly two neighbors of $v$. 

Suppose $x_1$ and $x_2$ are the only $A_3$-neighbors of $v$. 
Since $v$ has at least four pendent  $(3,5^-,5^-)$-faces, we may assume  $x_1y_1z_1$ is a $(3,5^-,5^-)$-face.
Putting $v$ in $A_3$ must not be an {\Ftf} of $G$, so we may assume $y_1\in A_3$, and furthermore $z_1\in A_4$. 
Now, since $z_1$ is a $5^-$-vertex, $(A_3+v-x_1, A_4+x_1)$ is an {\Ftf} of $G$, which is a contradiction. 

Suppose $x_1$ and $x_2$ are the only $A_4$-neighbors of $v$. 
Since $v$ has at least four pendent $(3,5^-,5^-)$-faces, we may assume  $x_1y_1z_1$ is a $(3,5^-,5^-)$-face.
Putting $v$ in $A_4$ must not be an {\Ftf} of $G$, so we may assume $y_1\in A_4$, and furthermore $z_1\in A_3$. 
Now, since both $y_1$ and $z_1$ are $5^-$-vertices, either $(A_3+x_1, A_4+v-x_1)$ or $(A_3+x_1-z_1, A_4+v-x_1+z_1)$ is an {\Ftf} of $G$, which is a contradiction. 
\end{proof}

\begin{lem}\label{lem:badNonSat22}
For $d\in\{6, \ldots, 10\}$, let $v$ be a $d$-vertex of $G$ incident with a terrible face $vxy$. Let $x'$ and $y'$ be the pendent neighbors of $x$ and $y$, respectively, where $x'$ is a $4^-$-vertex.
If $(A_3,A_4)$ is an {\Ftf} of $G-x$, then the following holds:
\begin{itemize}
    \item[\rm(i)] If $v$ is on either $(d-4)$ $3$-faces or $(d-5)$ terrible faces, then $x',y, y'\in A_4$, $v\in A_3$, and $v$ is $A_3$-saturated,
\item[\rm(ii)] If $v$ is  on $(d-4)$ $3$-faces, then $p$ and $q$ are in different parts whenever $pqv$ is a $3$-face, and all other neighbors (except $x$) of $v$ are in $A_3$. 
\end{itemize}
\end{lem}

\begin{proof} 
Since $x$ has only one $6^+$-neighbor $v$, Lemma~\ref{lem:coloring} (i) and (ii) implies that $v\in A_j$, $y, x'\in A_{7-j}$, and $v$ is $A_j$-saturated for some $j\in\{3, 4\}$.
Moreover, $y' \in A_{7-j}$, since putting $x$ in $A_{7-j}$ (and moving $x'$ from $A_{7-i}$ to $A_i$ if necessary) is not an {\Ftf} of $G$. 
To show (i), it is enough to prove $j=3$. 

Note that each $3$-face incident with $v$ is incident with at most one neighbor  of $v$ in $A_j$.
If $v$ is on $(d-4)$ 3-faces, then $v$ has at most $(d-5)+ (d-(2d-8))=3$ neighbors in $A_j$.
Since $v$ is $A_j$-saturated, this implies $j=3$. 
Moreover, this further implies (ii) holds.

If $v$ is on $(d-5)$ terrible faces, then $v$ has at most $(d-6)+ (d-(2d-10))=4$ neighbors in $A_j$.
Suppose to the contrary that $j=4$, so $v\in A_4$.
Since $v$ is $A_4$-saturated, every $A_3$-neighbor of $v$ is on a terrible face incident with $v$.   
Thus, $v$ has exactly $d-5$ neighbors in $A_3$, and all other neighbors except $x$ are in $A_4$. 

Let $Z$ be the set of $A_3$-neighbors of $v$ such that its pendent neighbor is also in $A_3$. 
Note that $y\in Z$.
If either $d\neq 10$ or $|Z|\geq 2$, then 
$(A_3+v-Z,A_4+Z+x-v)$ is an {\Ftf} of $G$, which is a contradiction.

Therefore, the only remaining case is when $d=10$ and $Z=\{y\}$. 
Let $vx_1y_1$, $vx_2y_2$, $vx_3y_3$, $vx_4y_4$ be the terrible faces incident with $v$ other than $vxy$. 
Let $x'_i$ and $y'_i$ be the pendent neighbors of $x_i$ and $y_i$, respectively. 
Since $Z=\{y\}$, we may assume $y_i\in A_3$ and  $x_i,y'_i$ are $A_4$-neighbors of $v$
 for each $i$.
Moreover, since neither $(A_3+x_i,A_4+x-x_i)$ nor $(A_3+x_i-x'_i,A_4+xx'_i-x_i)$ is an {\Ftf} of $G$, we conclude $x'_i$ is an $A_3$-saturated  
$5^+$-vertex.
Thus $y'_i$ is a $4^-$-vertex for every $i$, so $(A_3+v-y_1y_2y_3y_4,A_4+xy_1y_2y_3y_4-v)$ is an {\Ftf} of $G$, which is a contradiction.
\end{proof}

\begin{lem}\label{lem:rc-bad2face257}
In the graph $G$, the following holds:
\begin{itemize}
\item[\rm(i)] There is no bad vertex with only $6^-$-neighbors. {\rm{\ref{rc-bad2face257a}}}
\item[\rm(ii)] There is no $(F_2 \cup F_3)$-face with two bad vertices. {\rm{\ref{rc-bad2face257b}}}
\item[\rm(iii)] There is no $6$-vertex on three $3$-faces, one of which is a terrible face. {\rm\ref{rc-mediumd-3T3}}
\item[\rm(iv)] For $d \in \{6,\ldots ,10\}$, there is no $d$-vertex on $(d-4)$ terrible faces. {\rm\ref{rc-mediumd-4T3}}  
\item[\rm(v)]  For $d\in\{6, \ldots, 10\}$, there is no $d$-vertex $v$ on $(d-5)$ terrible faces, where $v$ has only $3^-$-neighbors.  {\rm\ref{rc-mediumd-5T3-6}}
\end{itemize}
\end{lem}

\begin{proof}
Recall that a bad vertex $v$ is on $(\deg(v)-4)$ $3$-faces, one of which is a terrible face. 
Let $v$ be a $d$-vertex on either $(d-4)$ $3$-faces or $(d-5)$ terrible faces, where one $3$-face is a terrible face $vxy$.
Let $x'$ and $y'$ be the pendent neighbors of $x$ and $y$, respectively, where $x'$ is a $4^-$-vertex.
By the minimality of $G$, there exists an {\Ftf} $(A_3,A_4)$ of $G-x$, and by Lemma~\ref{lem:badNonSat22}, we know $x', y,y' \in A_4$, $v\in A_3$, and $v$ is $A_3$-saturated.
Moreover, if $v$ is on $(d-4)$ $3$-faces, then $uv$ is incident with a $3$-face for every $A_4$-neighbor $u$ of $v$.
Let $X$ be the set of $A_4$-neighbors $u$ of $v$ such that $\deg(u) =3$,
$uv$ is incident with a $3$-face, and the pendent neighbor of $u$ is also in $A_4$.
Note that $y \in X$.

To show (i), suppose to the contrary that $v$ is a bad vertex with only $6^-$-neighbors.
Let $uvw$ be the non-terrible face incident with $v$, where $u$ is a $3^-$-vertex.
Since $(A_3+x+X-v, A_4+v-X)$ is not an {\Ftf} of $G$, we conclude that $w$ is an $A_4$-saturated  $6$-vertex.
Now, $(A_3+xw+X-v, A_4+v-X-w)$ is an {\Ftf} of $G$, which is a contradiction.

To show (ii), suppose to the contrary that $v$ is a bad vertex on a $(F_2 \cup F_3)$-face $uvw$, where $u$ is a $3^-$-vertex and $w$ is a bad vertex.
By Lemma~\ref{lem:badNonSat22}, $u$ and $w$ must be in different parts. 
Then $(A_3+x+X-v, A_4+v-X)$ is an {\Ftf} of $G$ since neither $u$ nor $w$ can be $A_4$-saturated, which is a contradiction.

Since  $v$ is not $A_3$-saturated in $G-x$ when $v$ is a $6$-vertex on three $3$-faces, (iii) holds.

To show (iv), suppose to the contrary that $v$ is on $(d-4)$ terrible faces. 
(This is possible only when $d\geq 2d-8$.)
Now, $(A_3+x+X-v, A_4+v-X)$ is an {\Ftf} of $G$, which is a contradiction.

To show (v), suppose to the contrary that $v$ has only $3^-$-neighbors and $v$ is on $(d-5)$ terrible faces $vx_1y_1({=}vxy), \ldots, vx_{d-5}y_{d-5}$, where $x'_i$ and $y'_i$ are the pendent neighbors of $x_i$ and $y_i$, respectively, and $x_i'$ is a $4^-$-vertex for all $i \in \{1, \ldots, d-5\}$.

First, suppose that every $A_4$-neighbor of $v$ is in $\{x_1, y_1, \ldots, x_{d-5}, y_{d-5}\}$.
Recall that 
$X$ is the set of $A_4$-neighbors $u$ of $v$ whose pendent neighbor is in $A_4$, and $y_1 \in X$.
We want that if a neighbor of $v$ is in $A_4$, then its pendent neighbor is in $A_3$. 
Let $Z = \{x_j \in A_4 \mid y_j \in A_4, x_j' \in A_3\}$, and $U = \{ x_j' \mid x_j \in Z \text{ and $x_j'$ is $A_3$-saturated}\}$.
Then $|(A_4\cap N_G(v) ) \setminus (X \cup Z)| \le d-6 \le 4$ so $(A_3+x_1+X+Z-v-U,A_4+v+U-X-Z)$ is an {\Ftf} of $G$, which is a contradiction. 
In particular, (v) holds when $d=10$.

Now assume that there is an $A_4$-neighbor $z$ of $v$ that is not in $\{x_1, y_1, \ldots, x_{d-5}, y_{d-5}\}$. 
Note that $N_G(v)=\{x_1, y_1, \ldots, x_{d-5}, y_{d-5}\}$ when $d=10$, so we may assume $d\in\{6, \ldots, 9\}$. 
Moreover, $v$ has $d-2(d-5) = 10-d$ neighbors $u$ where $uv$ is not on a terrible face and $|\{x_j,y_j\} \cap A_3| \le 1$ for all $j$.
Since $v$ is $A_3$-saturated, $(A_4\cap N_G(v) ) \setminus \{x_1, y_1, \ldots, x_{d-5}, y_{d-5}\}=\{z\}$, and $x_j$ and $y_j$ are in the different parts of $(A_3,A_4)$ for all $j \in \{2, \ldots, d-5\}$.
This implies $|(A_4\cap N_G(v)) \setminus X| \le 1 + (d-5) -1 = d-5 \le 4$ as $y_1\in X$ and $d \in \{6,\ldots,9\}$.
Since neither $(A_3+x_1+X-v,A_4+v-X)$ nor $(A_3+x_1z+X-v,A_4+v-X-z)$ is an {\Ftf} of $G$, we know $z$ is $A_4$-saturated and has an $A_3$-neighbor other than $v$.
Thus $u$ is a $6^+$-vertex, which is a contradiction.
Hence, (v) holds.
\end{proof}

\begin{lem}\label{lem:rc-mediumT3H3}
In the graph $G$, for $d\in\{7, \ldots, 10\}$, there is no non-bad $d$-vertex $v$ on an $F^*_2$-face, where $v$ is on $(d-6)$ other $3$-faces, each of which is either terrible or in $F^*_2$.
\rm{\ref{rc-mediumT3H3}}
\end{lem}

\begin{proof} 
Suppose to the contrary that a non-bad $d$-vertex $v$ is on an $F^*_2$-face $vx_1y_1$ and each of the other $(d-6)$ $3$-faces $vx_2y_2, \ldots, vx_{d-5}y_{d-5}$ on $v$ is either terrible or in $F^*_2$.
For each $i\in\{1, \ldots, d-5\}$, if $vx_iy_i$ is an $F^*_2$-face, then let $x_i$ be the $2$-vertex, so $y_i$ is either a $5$-vertex or a bad vertex.
If $vx_iy_i$ is a terrible face, then let $x'_i$ and $y'_i$ be the pendent neighbors of the $3$-vertices $x_i$ and $y_i$, respectively, and let $x'_i$ be a $4^-$-vertex.

By the minimality of $G$, there exists an {\Ftf} of $G-x_1$; let $(A_3, A_4)$ be such a partition where $|\{x_2,...,x_{d-5}\}\cap A_3|$ is maximized.  
Since $y_1$ cannot be $A_4$-saturated, by Lemma~\ref{lem:coloring}, $y_1$ is an $A_3$-saturated vertex in $A_3$ and $v$ is an $A_4$-saturated vertex in $A_4$. 
Also, every $A_4$-neighbor of $y_1$ except $v$
is on a terrible face incident with $y_1$. 
Moreover, $v$ has at most $d-2(d-5) = 10-d$ neighbors $u$ such that $uv$ is not incident with a $3$-face and $|\{x_j,y_j\} \cap A_4| \le 1$ for all $j \in \{2, \ldots, d-5\}$.
Since $v$ is $A_4$-saturated, we know $x_j$ and $y_j$ are in different parts of $(A_3, A_4)$ for all $j \in \{2, \ldots, d-5\}$, and $v$ has exactly $(10-d)$ neighbors $u$ where $uv$ is not incident with a $3$-face, and all such neighbors are in $A_4$.

\begin{clm} \label{A_3inYandZ}
For $j \in \{2,\ldots,d-5\}$, if $vx_jy_j$ is a terrible face, then $x_j,y_j'\in A_3$ and $y_j\in A_4$.
\end{clm}
\begin{proof}
Note that $x_j$ and $y_j$ are in different parts. 
Suppose that $x_j\in A_4$ and $y_j\in A_3$.
Since $(A_3+x_j, A_4+x_1-x_j)$ is not an {\Ftf} of $G$, we have $x'_j\in A_3$.
Moreover, either $x'_j$ is $A_3$-saturated or $y'_j\in A_3$.
If $x'_j$ is $A_3$-saturated, then 
$(A_3-x'_j+x_j,A_4+x'_j+x_1-x_j)$ is an {\Ftf} of $G$, which is a contradiction. 
Hence,  $x'_j$ is not $A_3$-saturated, and so $y'_j\in A_3$. 
Now, $(A_3-y_j+x_j,A_4+y_j-x_j)$ is an {\Ftf} of $G-x_1$, which is contradiction to the choice of $(A_3,A_4)$.
Therefore, $x_j\in A_3$ and $y_j\in A_4$.
If $y'_j\in A_4$, then
$(A_3+y_j, A_4+x_1-y_j)$ is an {\Ftf} of $G$, which is a contradiction.  Thus $y'_j\in A_3$. \end{proof}

\begin{clm}\label{W:A_3-sat}
For $j\in\{1, \ldots, d-5\}$, let $vx_jy_j$ be an $F^*_2$-face.
If $y_j\in A_3$, then $y_j$ is $A_3$-saturated.
\end{clm}
\begin{proof}
It is clear for $j=1$.
Suppose that $j\ge 2$. 
If $y_j$ is not $A_3$-saturated, then $(A_3+x_j,A_4+x_1-x_j)$ is an {\Ftf} of $G$, which is a contradiction. Thus $y_j$ is $A_3$-saturated.
\end{proof}

For simplicity, define the following sets:
\begin{eqnarray*}
T &=& \{ x_j  \mid \text{$vx_jy_j$ is a terrible face}\},\\
U &=& \{x_j\in A_3 \mid \text{$vx_jy_j$ is an $F^*_2$-face}\},\\
W &=& \{y_j \in A_3 \mid \text{$vx_jy_j$ is an $F^*_2$-face}\}.
\end{eqnarray*}


Since every vertex in $W$ is $A_3$-saturated by Claim \ref{W:A_3-sat}, as long as $y_j\in W$ is a bad vertex, every $A_4$-neighbor of $y_j$ except $v$  
is on a terrible face incident with $y_j$.
See Figure~\ref{fig:bad6,7,8}. 
For a bad vertex $y_j\in W$, define 
\[X_j =\{u \in A_4 \cap N_G(y_j) \mid \text{$u$ and $y_j$ are on a terrible face  and the pendent neighbor of $u$ is in } A_4\}.\]
In addition, let $X$ be the union of all such $X_j$'s.
Together with Claim \ref{A_3inYandZ}, 
$(A_3+x_1v+X -y_1-T- U - W,A_4+y_1+T+ U + W-v-X)$ is an {\Ftf} of $G$, which is a contradiction.
\end{proof}

\section{Future Research Directions}\label{sec:future}

We proved that planar graphs without $4$-cycles and $5$-cycles have an \Ftf.
It is known that for every integer $k$, there exists a planar graph without $4$-cycles and $5$-cycles that does not have an \FF{1}{k}, since there exists a planar graph without $4$-cycles and $5$-cycles that does not have an \DeltaDelta{1}{k}~\cite{SN2018} as mentioned in the introduction.
It would be interesting to determine if there exists an integer $d_2$ such that planar graphs without $4$-cycles and $5$-cycles have an \FF{2}{d_2}. 
We pose the only remaining case as the following question:

\begin{ques}
Does there exist an integer $d_2$ such that every planar graph without $4$-cycles and $5$-cycles has an \FF{2}{d_2}?
\end{ques}

Actually, it is even unknown when the second part is allowed to be a forest of unbounded degree, namely, an \FF{2}{\infty}.
Moreover, it is not determined if every planar graph without $4$-cycles and $5$-cycles has an \FF{1}{\infty}.
This is our second question. 

\begin{ques}
Is it true that every planar graph without $4$-cycles and $5$-cycles has an \FF{1}{\infty}?
\end{ques}




\section*{Acknowledgements}
Ilkyoo Choi was supported by the Basic Science Research Program through the National Research Foundation of Korea (NRF) funded by the Ministry of Education (NRF-2018R1D1A1B07043049), and also by the Hankuk University of Foreign Studies Research Fund.
Boram Park was supported by Basic Science Research Program through the National Research Foundation of Korea (NRF) funded by the Ministry of Science, ICT and Future Planning (NRF-2018R1C1B6003577).

\bibliography{ref}{}
\bibliographystyle{plain}
 
\end{document}